\documentclass[a4paper]{article}
\usepackage{amsfonts}
\usepackage{amsmath}
\usepackage{mathrsfs}
\usepackage{graphicx}
\usepackage{amsmath,amsthm,amssymb,amscd}

\renewcommand{\baselinestretch}{1.05}

\renewcommand{\baselinestretch}{1.05}

\newtheorem{Thm}{Theorem}[section]
\newtheorem{Lem}[Thm]{Lemma}

\theoremstyle{remark}

\theoremstyle{claim}

\newcommand{\eps}{\varepsilon}
\newcommand{\R}{\mathbb{R}}
\newcommand{\N}{{\mathbb{N}}}

\newcommand{\htop}{h_{\text{top}}}

\newcommand{\E}{\mathcal{E}}

\begin{document}

\title{\bf Multifractal Analysis of The New Level sets}
\author{{Yiwei Dong$^{*}$ and Xueting Tian$^{**}$}\\
{\em\small School of Mathematical Science,  Fudan University}\\
{\em\small Shanghai 200433, People's Republic of China}\\
{\small $^*$Email:dongyiwei0@sina.com;   $^{**}$Email:xuetingtian@fudan.edu.cn}\\
}
\date{}
\maketitle

\renewcommand{\baselinestretch}{1.2}
\large\normalsize

\footnotetext {AMS Review: 37A35; 37B05; 37C45.  }

\begin{abstract}
By an appropriate definition, we divide the irregular set into level sets. Then we characterize the multifractal spectrum of these new pieces by calculating their entropies. We also compute the entropies of various intersections of the level sets of regular and irregular set which is rarely studied in the literature. Moreover, our conclusions also hold for the topological pressure. Finally, we consider the continuous case and use our results to give a description for the suspension flow.

\end{abstract}

\section{Introduction}
By a dynamical system $(X,T)$, we mean a continuous map $T$ acting on a compact metric space $(X,d)$.
In the classical framework for multifractal analysis of dynamical systems \cite{BPS,Pesin}, one begins with a local asymptotic quantity $\phi(x):X\to \R^d$ and defines the level sets $K(\alpha):=\{x\in X:\phi(x)=\alpha\}$. These $K(\alpha)$ form a multifractal decomposition and the
function $\alpha\mapsto \dim K(\alpha)$ is a multifractal spectrum. Here $\dim K(\alpha)$ is a global dimensional quantity that assigns to each level set of $\phi$ a `size' or `complexity',
such as its topological entropy or Hausdorff dimension.

For concrete, let us consider the ergodic average $\frac{1}{n}\sum_{i=0}^{n-1}\varphi(T^{i}(x))$ of a continuous function $\varphi\colon X\to \R$ along the orbit of some $x\in X$. By the classical Birkhoff ergodic theorem \cite{Walters}, if $\mu$ is a $T-$invariant measure, then for $\mu-$a.e. $x\in X$, such average converges. We call the set of such points $\varphi-$regular points and denote them by $R_{\varphi}$. Correspondingly, the complementary set is called $\varphi$-irregular set and we denote them by $I_{\varphi}$. The union $I(f)=\bigcup_{\varphi\in C(X)}$ is called `the set of points with historic behaviour' by Ruelle \cite{Ruelle}. The idea is that points for which the Birkhoff average does
not exist are capturing the `history' of the system, whereas points whose Birkhoff
average converge only see average behaviour.

For $x\in R_{\varphi}$, let $\phi(x)=\lim_n\frac{1}{n}\sum_{i=0}^{n-1}\varphi(T^{i}(x))$. We are in a position to give the multifractal analysis for the level sets $K(\alpha)$. Indeed, this is done by many authors. In the early time, the thermodynamic approach was widely used. For example, Fan and Feng \cite{Fan-Feng} considered symbolic systems and related the
pressure of an interaction function $\Phi$ to its long-term time averages through the
Hausdorff and packing dimensions of the subsets on which $\Phi$ has prescribed
long-term time-average values. Moreover, Barreira
\emph{et al} \cite{Barreira-Schmeling2002} established a higher-dimensional version of multifractal analysis for several classes of hyperbolic dynamical systems. Later on, Takens and Verbitskiy \cite{Takens-Verbitskiy} worked under the condition of specification property and provided a orbit-gluing method to build the variational principle for $K(\alpha)$.

From the viewpoint of the ergodic theory, the regular part seems to dominate the behavior of the systems. Nevertheless, various results show that the irregular part can also exhibit a rich structure. For example, Fan \emph{et al} \cite{FFW} paid attention to symbolic dynamics and showed that $I_{\varphi}(f)$ either is  empty or has full topological entropy. That is, $I_{\varphi}(f)$ can be as complex as the whole system. Later on, Barreira and Schmeling \cite{Barreira-Schmeling2000} studied a broad class of systems including conformal repellers and horseshoes and demonstrated that the irregular set carries full entropy and Hausdoff dimension. Moreover, for systems with the specification property, Chen \emph{et al} \cite{CTS} proved that the irregular set $I(f)$ has full topological entropy.

It is known that the topological mixing subshift of finite type has the specification property. It is also known that 
a compact locally maximal hyperbolic set of a topological mixing diffeomorphism has Bowen's specification property \cite{Bowen1971}. Besides, the continuous topologically mixing map of the
interval has Bowen¡¯s specification property \cite{Bl}. Thus, all the above mentioned systems satisfy the specification property.
However, in this paper, we concentrate on systems with the almost specification property which was introduced in \cite{PS2005} and later renamed by Thompson \cite{Thompson2012}. The reason lies in that the specification property implies the almost specification property \cite{PS2007}. Moreover, it is known that every $\beta$-shift satisfies the almost specification property \cite{PS2005} while the set of $\beta$ for which the $\beta$-shift has the specification property
has zero Lebesgue measure \cite{Buzzi}. Therefore, the almost specification property is a non-trivial generalization of the specification property. So we can work under a more general circumstance.

\subsection{Level sets for irregular points and the various variational principles}
Let us begin with some definitions.
For any continuous function $\varphi:X\rightarrow \mathbb{R}$, we define
\begin{displaymath}
\underline{\varphi}(x):=\liminf_{n\rightarrow\infty}\frac{1}{n}\sum_{i=0}^{n-1}\varphi(T^{i}(x))~~\textrm{and}~~\overline{\varphi}(x):=\limsup_{n\rightarrow\infty}\frac{1}{n}\sum_{i=0}^{n-1}\varphi(T^{i}(x)).
\end{displaymath}
Then we can divide $X$ into level sets:
\begin{displaymath}
X=\bigsqcup_{c\leq d} X_{\varphi}(c,d),
\end{displaymath}
where
\begin{equation}\label{def-level-set}
  X_{\varphi}(c,d):=\{x\in X~|~\underline{\varphi}(x)=c~\textrm{and}~\overline{\varphi}(x)=d\}.
\end{equation}
Define
$$R_{\varphi}(a):=X_{\varphi}(a,a),~~R_{\varphi}:=\bigsqcup_{a\in \mathbb{R}}R_{\varphi}(a)~~\textrm{and}~~I_{\varphi}:=\bigsqcup_{c<d}X_{\varphi}(c,d).$$
One can easily check that $R_{\varphi}$ and $I_{\varphi}$ defined here coincide with the classical regular and irregular set.

Now suppose $(X,T)$ satisfies the almost specification property. In \cite{Thompson2012}, Thompson showed that $I_{\varphi}$ either is empty or has full entropy. Here we have a finer description.
\begin{Thm}\label{thm-inter-reg-irreg}
For any $\varphi_{1}, \varphi_{2}\in C(X)$, $R_{\varphi_{1}}\cap I_{\varphi_{2}}$ has full entropy or is empty.
\end{Thm}
For the level set of $\varphi$-regular points, Takens and Verbitskiy \cite{Takens-Verbitskiy} showed the following variational principle:
$$\htop(T,R_{\varphi}(a))=\sup~\left\{h_{\mu}(T)~|~\mu\in M(T,X)~and~\int\varphi d\mu=a\right\}.$$
Here, we also have further result.
\begin{Thm}\label{thm-inter-level-irreg}
For any $\varphi_{1}, \varphi_{2}\in C(X)$, if $R_{\varphi_{1}}(a)\cap I_{\varphi_{2}}\neq\emptyset$, then
$$\htop\left(T,R_{\varphi_{1}}(a)\cap I_{\varphi_{2}}\right)=\sup~\left\{h_{\mu}(T)~|~\mu\in M(T,X)~and~\int\varphi_{1} d\mu=a\right\}.$$
\end{Thm}
Now it is time to illustrate our new defined level sets for the irregular points.
\begin{Thm}\label{thm-level-of-irreg}
For any real numbers $c\leq d$, either $X_{\varphi}(c,d)$ is empty or
\begin{displaymath}
\htop\left(T,X_{\varphi}(c,d)\right)=\min_{\xi=c,d}~\sup~\left\{h_{\mu}(T)~|~\mu\in M(T,X)~and~\int\varphi d\mu=\xi\right\}.
\end{displaymath}
\end{Thm}
Furthermore, we have

\begin{Thm}\label{thm-multifractal}
Let $\{\varphi_i\}_{i=1}^n$ and $\psi$ be continuous functions on $X$. Then for any real numbers $\{a_i\}_{i=1}^n$ and $c\leq d$, either $X_{\psi}(c,d)$ is empty or
\begin{eqnarray*}
   & & \htop\left(T,(\bigcap_{i=1}^nR_{\varphi_{i}}(a_i))\cap X_{\psi}(c,d)\right) \\
   &=& \min_{\xi=c,d}~\sup~\left\{h_{\mu}(T)~|~\mu\in M(T,X)~s.t.~\int\varphi_{i}d\mu=a_i,~i=1,\ldots,n~and~\int\psi d\mu=\xi\right\}.
\end{eqnarray*}

\end{Thm}
The proof relies heavily on the following theorem which we now state.

\subsection{Variational principle for the finite convex combination of invariant measures}
In \cite{PS2007}, Pfister and Sullivan used the almost specification property and the uniform separation property to give a variational principle for the saturated sets. However, we observe that for some special $K$s, the uniform separation condition naturally holds.
\begin{Thm}\label{thm-var-prin}
If $K=\{\sum_{i=1}^{n}t_{i}\lambda_{i}:t_{i}\in[0,1],\sum_{i=1}^{n}t_{i}=1,i=1,\cdots,n\}$
for some finite subset $\{\lambda_i\}_{i=1}^n\subset M(T,X)$, then
$$\htop(T,G_K)=\min\{h_{\lambda_i}(T),i=1,\ldots,n\}.$$
\end{Thm}

This paper is organized as follows. In section \ref{pre}, we give the basic knowledge. After that, we give the proof of the technical theorem \ref{thm-var-prin} in section \ref{technical-thm}. Then we take use of this result to deduce our main theorem \ref{thm-inter-reg-irreg} through \ref{thm-multifractal} in section \ref{proof-for-entropy}. Furthermore, we point out that the conclusions also hold for topological pressure in section \ref{about-pressure}. Finally, we consider a continuous case in section \ref{suspension-flow}, namely, we apply our results to the suspension flow.



\section{Preliminaries}\label{pre}
\subsection{Some notions}
Throughout this paper, we denote by $C(X)$ the set of continuous maps on $X$ and $M(X)$ the set of Borel probability measures on $X$. We say $\mu\in M(X)$ is a $T$-invariant measure if for any Borel measurable set $A$, one has $\mu(A)=\mu(T^{-1}A)$. The set of $T$-invariant measures are denoted by $M(T,X)$. We write $h_{\mu}(T)$ for the metric entropy of $\mu\in M(T,X).$

Define the empirical measure of $x\in X$ as
\begin{equation*}
 \mathcal{E}_{n}(x):=\frac{1}{n}\sum_{j=0}^{n-1}\delta_{f^{j}(x)},
\end{equation*}
where $\delta_{x}$ is the Dirac mass at $x$. The limit-point set of $\{\mathcal{E}_{n}(x)\}$ is always a non-empty compact connected subset $V_{T}(x)\subset M(T,X)$. A subset $D\subset X$ is called saturated if for any $x\in D$, one has
$$\{y\in X: V_T(y)=V_T(x)\}\subset D.$$
Of particular interest are the points $x\in X$ such that $V_{T}(x)=\{\mu\}$. These points are called generic points of $\mu$ and we denote them by $G_{\mu}$. More generally, if $K\subset M(T,X)$ is a non-empty compact connected subset, we denote
by $G_{K}$ the saturated set of points $x$ such that $V_{T}(x)=K$.

If ${\{\varphi_{j}\}}_{j\in\mathbb{N}}$ is a dense subset of $C(X)$, then
$$\rho(\nu,\tau)=\sum_{j=1}^{\infty}\frac{|\int\varphi_{j}d\nu-\int\varphi_{j}d\tau|}{2^{j}\|\varphi_{j}\|}$$
defines a metric on $M(X)$ for the $weak^{*}$ topology  \cite{Walters}. Note that
$$\rho(\nu,\tau)\leq 2,~\forall~\nu,\tau\in M(X).$$
In the rest of this article, we follow Pfister and Sullivan \cite{PS2007} and  use an equivalent metric on $X$, namely,
$d(x,y):=\rho(\delta_{x},\delta_{y}).$
We denote a ball in $M(X)$ by
$$\mathcal{B}(\nu,r):=\{\rho(\nu,\mu)<r:\mu\in M(X)\}.$$
For $\delta>0$ and $\eps>0$, two points $x$ and $y$ are $(\delta,n,\eps)$-separated if
$$|\{j:d(T^{j}(x),T^{j}(y))>\eps,~0\leq j\leq n-1\}|\geq\delta n.$$
Where $|B|$ denotes the cardinality of $B$. A subset $E$ is $(\delta,n,\eps)$-separated if any pair of different points of $E$ are $(\delta,n,\eps)$-separated.
Let $F\subset M(T,X)$ be a neighborhood. We define
$X_{n,F}:=\{x\in X:\mathcal{E}_{n}(x)\in F\}$
and
$$N(F;\delta,n,\eps):=~\textrm{maximal cardinality of a} ~(\delta,n,\eps)-\textrm{separated subset of}~ X_{n,F}.$$

\subsection{The specification property}
A continuous map $T:X\rightarrow X$ satisfies the \emph{specification} property if for all $\eps>0$,  there exists an integer $m(\eps)$ such that for any collection $\{I_{j}=[a_{j},b_{j}]\subset\mathbb{Z}^{+}:j=1,\cdots,k\}$ of finite intervals with $a_{j+1}-b_{j}\geq m(\eps)$ for $j=1,\cdots,k-1$ and any $x_{1},\cdots,x_{k}$ in $X$, there exists a point $x\in X$ such that
$$d(T^{a_{j}+t}x,T^{t}x_{j})<\eps $$ 
for all $t=0,\cdots,b_{j}-a_{j}$ and $j=1,\cdots,k$. Furthermore, we say $f$ satisfies \emph{Bowen's specification} property if the point $x$ can be chosen to be periodic with period $p$ for every $p\geq b_{k}-a_{1}+m(\epsilon)$.

\subsection{The almost specification property}
Let $\eps_0>0$. A function $g:\mathbb{N}\times (0,\eps_0)\to \mathbb{N}$ is called a \emph{mistake function} if for all $\eps\in (0,\varepsilon_0)$ and all $n\in \mathbb{N}$, $g(n,\eps)\leq g(n+1,\eps)$ and
$$\lim_{n}\frac{g(n,\eps)}{n}=0.$$
If $\eps\geq \eps_0$, we define $g(n,\eps)=g(n,\eps_0)$.
For $n\in \mathbb{N}$ large enough such that $g(n,\eps)<n$, we define
$$I(g;n,\eps):=\{\Lambda\subset\{0,1,\cdots,n-1\}:|\Lambda|\geq n-g(n,\eps)\}.$$
For a finite set of indices $\Lambda\subset\{0,1,\cdots,n-1\}$, we define the \emph{Bowen distance} of $x,y\in X$ along $\Lambda$ by
$$d_{\Lambda}(x,y):=\max_{j\in \Lambda}\{d(T^jx,T^jy)\}$$
and the \emph{Bowen ball} of radius $\eps$ centered at $x$ by
$$B_{\Lambda}(x,\eps):=\{y\in X:d_{\Lambda}(x,y)<\eps\}.$$
When $g(n,\eps)<n$, we define the $(g;n,\eps)$-Bowen ball centered at $x$ as
$$B_n(g;x,\eps):=\{y\in X:y\in B_{\Lambda}(x,\eps)~\textrm{for some}~\Lambda\in I(g;n,\eps)\}=\bigcup_{\Lambda\in I(g;n,\eps)}B_{\Lambda}(x,\eps).$$

The dynamical system $(X,T)$ has \emph{the almost specification} property with mistake function $g$, if there exists a function $k_g:(0,+\infty)\to \N$ such that for any
$\eps_{1}>0,\cdots,\eps_{m}>0$, any points $x_{1},\cdots,x_{m}\in X$, and any integers $n_{1}\geq k_g(\eps_{1}),\cdots,n_{m}\geq k_g(\eps_{m})$, we can find a point $z\in X$ such that
\begin{equation*}
  T^{l_{j}}(z)\in B_{n_{j}}(g;x_{j},\eps_{j}),~j=1,\cdots,m,
\end{equation*}
where $n_{0}=0~\textrm{and}~l_{j}=\sum_{s=0}^{j-1}n_{s}$.

\subsection{Bowen's entropy for non-compact sets}\label{entropy-for-maps}
This is due to Bowen \cite{Bowen} and we state it in a more convenient form. For any $x\in X$ and $\eps>0$, define
$$B_n(T,x,\eps) = \{y\in X:d(T^ix,T^iy)<\eps,i=0,\ldots,n-1\},n\in\N.$$
Let $Z\subset X$ be an arbitrary Borel set, not necessarily compact or invariant.
Let $s\in \R$ and
\begin{eqnarray*}
  P(Z,s,\Lambda) &=& \sum_{B_{n_i}(T,x_i,\eps)\in\Lambda}\exp\left(-sn_i\right), \\
  C(Z,s,\eps,N) &=& \inf_{\Lambda}Q(Z,s,\Gamma),
\end{eqnarray*}
where the infimum is taken over all finite or countable collections of the form $\Lambda=\{B_{n_i}(T,x_i,\eps)\}_i$, with $x_i\in X$ such that $\Lambda$ covers $Z$ and $n_i\geq N$ for all $i\geq1$.

Since $C(Z;s,\eps,N)$ is non-decreasing in $N$, we define
$$C(Z;s,\eps):=\lim_{N\rightarrow\infty}C(Z;s,\eps,N).$$
Then we define
$$\htop(T,Z;\eps) := \inf\{s:C(Z;s,\eps)=0\}=\sup\{s:C(Z;s,\eps)=\infty\}.$$
Finally, the \emph{topological entropy} of $Z$ with respect to $T$ is
$$\htop(T,Z):=\lim_{\eps\rightarrow0} \htop(T,Z;\eps).$$

\section{Variational principle for a finite convex  combination of invariant measures}\label{technical-thm}
First, let us point out an important fact.
\begin{Lem}\label{rational-convex-combination}
For any $\eta>0$, there exist $\delta^*>0$ and $\eps^*>0$ so that for any $\lambda_i$ and any neighborhood $F_i\subset M(X)$ of $\lambda_i$, there exists $n^*_{F_i,\lambda_i,\eta}$ such that
\begin{equation}\label{N-lambda-i}
  N(F_i;\delta^*,n,\eps^*)\geq 2^{n(h_{\lambda_i}(T)-\eta)}~\textrm{if}~n\geq n^*_{F_i,\lambda_i,\eta}.
\end{equation}
\end{Lem}
\begin{proof}
Since $(X,f)$ has the almost specification property, we have by \cite[Proposition 6.1]{PS2007} that
$$s(\nu)=h(T,\nu)~\textrm{if}~\nu\in M(T,X),$$
where $s(\nu)=\lim\limits_{\eps\to 0}\lim\limits_{\delta\to 0}\inf\limits_{F\ni\mu}\lim\limits_{n\to\infty}\frac{1}{n}\log_2N(F;\delta,n,\eps)$.
Therefore, for any $\eta>0$, there exist $\delta^*_i>0$ and $\eps^*_i>0$ so that for any $\lambda_i$ and any neighborhood $F_i\subset M(X)$ of $\lambda_i$, there exists $n^*_{F_i,\lambda_i,\eta}$ such that
$$
  N(F_i;\delta^*_i,n,\eps^*_i)\geq 2^{n(h_{\lambda_i}(T)-\eta)}~\textrm{if}~n\geq n^*_{F_i,\lambda_i,\eta}.
$$

On the other hand, one observes that $N(F;\delta,n,\eps)$ is non-increasing in $\delta$ and $\eps$, so if we choose
$$\delta^*=\min\{\delta^*_i:i=1,\ldots,n\}>0~\textrm{and}~\eps^*=\min\{\eps^*_i:i=1,\ldots,n\}>0,$$
then one sees that \eqref{N-lambda-i} holds.
\end{proof}
Based on this observation, we show further that the separation is uniform on the finite rational convex combination of the invariant measures.
\begin{Lem}\label{lem-rational-covex}
For any $\eta>0$, there exist $\delta^*>0$ and $\eps^*>0$ so that if $\lambda$ is a rational convex combination of $\{\lambda_i\}_{i=1}^n$, i.e. there exist $\{c_i\}_{i=1}^n\subset\N$  such that $\lambda=\sum_{i=1}^n\frac{c_i}{c}\lambda_i$ with $c=\sum_{i=1}^nc_i$, then for any neighborhood $F\subset M(X)\cap K$ of $\lambda$, there exists $n^*_{F,\lambda,\eta}$ such that
\begin{equation}\label{N-lambda}
  N(F;\delta^*,n,\eps^*)\geq 2^{n(h_{\lambda}(T)-\eta)}~\text{if}~n\geq n^*_{F,\lambda,\eta}.
\end{equation}

\end{Lem}
\begin{proof}
By lemma \ref{rational-convex-combination}, for any $\eta>0$, there exist $\widetilde{\delta}^*>0$ and $\widetilde{\eps}^*>0$ so that for any $\lambda_i$ and any neighborhood $F_i\subset M(X)$ of $\lambda_i$, there exists $n^*_{F_i,\lambda_i,\frac{\eta}{2}}$ such that
\begin{equation}
  N(F_i;\widetilde{\delta}^*,n,\widetilde{\eps}^*)\geq 2^{n(h_{\lambda_i}(T)-\frac{\eta}{2})}~\textrm{if}~n\geq n^*_{F_i,\lambda_i,\frac{\eta}{2}}.
\end{equation}
For any neighborhood $F\subset M(X)\cap K$ of $\lambda$, there exists a $r>0$ such that $(\mathcal{B}(\lambda,r)\cap K)\subset F$. Choose each $F_i=\mathcal{B}(\lambda_i,\frac{r}{2})$. Let $\eps=\min\{\frac{\widetilde{\eps}^*}{4},\frac{r}{6}\}$ and $\widetilde{n}\in\N$ be such that
$$\frac{g(n,\eps)}{n}\leq \min\left\{\frac{\widetilde{\delta^*}}{8},\frac{r}{12}\right\},~\forall n\geq\widetilde{n}.$$
Now set
$$\widehat{n}=\max_{i=1,\ldots,n}\left\{n^*_{F_i,\lambda_i,\frac{\eta}{2}},\widetilde{n},k_g(\eps),\frac{12}{r},\frac{2h_{\lambda}(T)}{\eta}\right\}.$$
One notes that $\widetilde{\eps}^*,\widetilde{\delta}^*$ only depend on $\eta$. One also notes that $r$ depends on $\lambda,F$. Moreover, $\{\lambda_i\}_{i=1}^n$ is fixed and thus $F_i$ only depends on $r$. So $\widehat{n}$ only depnds on $F,\lambda,\eta$. We prove that lemma \ref{lem-rational-covex} holds for $\delta^*=\frac{\widetilde{\delta}^*}{4}$, $\eps^*=\frac{\widetilde{\eps}^*}{2}$ and $n^*_{F,\lambda,\eta}=c\widehat{n}$.


Actually, for any $n\geq c\widehat{n}$, there exists a $l\in\N$ such that $c(\widehat{n}+l)\leq n<c(\widehat{n}+l+1)$. By lemma \ref{rational-convex-combination}, there exist $(\widetilde{\delta}^*,\widehat{n}+l,\widetilde{\eps}^*)$-separated sets $\Gamma_i$, $i=1,\ldots,n$ such that for any $x_i\in\Gamma_i$, $\rho(\mathcal{E}_{\widehat{n}+l}(x_i),\lambda_i)<r$. Moreover, one has
\begin{equation}\label{estimation-of-Gamma-i}
  |\Gamma_i|\geq 2^{(\widehat{n}+l)(h_{\lambda_i}(T)-\frac{\eta}{2})}.
\end{equation}

Consider the product space $\prod_{i=1}^n\Gamma_i^{c_i}$ whose elements are $\underline{x}=(x_{1,1},\ldots,x_{1,c_1}.\ldots,x_{n,1},\ldots,x_{n,c_n})$ with $$x_{i,j}\in\Gamma_i,~i=1,\ldots,n,~j=1,\ldots,c_i.$$
For any $\underline{x}\in\prod_{i=1}^n\Gamma_i^{c_i}$, by the almost specification property, there exists an $\widetilde{x}\in X$ such that
$$T^{l_{i,j}}(z)\in B_{\widehat{n}+l}(g;x_{i,j},\frac{\widetilde{\eps}^*}{4}),$$
where $l_{i,j}=\sum_{k=1}^{i-1}c_k+(j-1)c_i$. For simplicity, we say that $\widetilde{x}$ \emph{almost traces} $\underline{x}$.

Now we claim that for any $\underline{x},\underline{y}\in\prod_{i=1}^n\Gamma_i^{c_i}$ with $\underline{x}\neq\underline{y}$, the corresponding almost tracing points $\widetilde{x},\widetilde{y}\in X$ are $(\delta^*,n,\eps^*)$-separated.

In fact, it is not hard to see that there exist at least
$$n'=c(\widehat{n}+l)\widetilde{\delta}^*-2cg(\widehat{n}+l)$$
places in the orbits of $\widetilde{x}$ and $\widetilde{y}$ which differs by $\widetilde{\eps}^*-2\cdot\frac{\widetilde{\eps}^*}{4}=\frac{\widetilde{\eps}^*}{2}=\eps^*$.
An easy calculation shows that
$$\frac{n'}{n}\geq\frac{c(\widehat{n}+l)\widetilde{\delta}^*-2cg(\widehat{n}+l,\eps)}{c(\widehat{n}+l+1)}\geq \frac{\widetilde{\delta}^*}{2}-2\cdot \frac{\widetilde{\delta}^*}{8}=\frac{\widetilde{\delta}^*}{4}=\delta^*.$$

Therefore, if we define
$$S_{F,\lambda,\eta}=\left\{\widetilde{x}\in X\colon \widetilde{x}~\textrm{almost traces}~\textrm{some}~\underline{x}\in \prod_{i=1}^n\Gamma_i^{c_i}\right\},$$
one sees that $S_{F,\lambda,\eta}$ is a $(\delta,n,\eps^*)$-separated set.

Moreover, for any $\widetilde{x}\in S_{F,\lambda,\eta}$, we have the following estimations
\begin{eqnarray*}
  \rho\left(\mathcal{E}_n(\widetilde{x}),\mathcal{E}_{c(\widehat{n}+l)}(\widetilde{x})\right) &\leq& \frac{n-c(\widehat{n}+l)}{n}\leq \frac{2}{\widehat{n}}\leq \frac{r}{6} \\
  \rho\left(\mathcal{E}_{c(\widehat{n}+l)}(\widetilde{x}),\frac{\sum_{i,j}\mathcal{E}_{\widehat{n}+l}(x_{i,j})}{c(\widehat{n}+l)}\right) &\leq& \left(1-\frac{g(\widehat{n}+l,\eps)}{\widehat{n}+l}\right)\cdot\eps+\frac{g(\widehat{n}+l,\eps)}{\widehat{n}+l}\cdot2\leq \frac{r}{6}+\frac{r}{6}=\frac{r}{3} \\
  \rho\left(\frac{\sum_{i,j}\mathcal{E}_{\widehat{n}+l}(x_{i,j})}{c(\widehat{n}+l)},\lambda\right) &\leq& \sum_{i=1}^n\frac{c_i}{c}\cdot\frac{r}{2}=\frac{r}{2}
\end{eqnarray*}
Hence, we have $\mathcal{E}_n(\widetilde{x})\in \mathcal{B}(\lambda,r)\subset F$, which implies that $S_{F,\lambda,\eta}\subset X_{n,F}$. Then it is left to show that
$$|S_{F,\lambda,\eta}|= \prod_{i=1}^n|\Gamma_i|^{c_i}\geq 2^{n(h_{\lambda}(T)-\eta)}.$$

By \eqref{estimation-of-Gamma-i}, one has
$$\prod_{i=1}^n|\Gamma_i|^{c_i}\geq 2^{\sum_{i=1}^nc_i(\widehat{n}+l)(h_{\lambda_i}(T)-\frac{\eta}{2})}.$$
By the affinity of metric entropy,
$$h_{\lambda}(T)=h\left(T,\sum_{i=1}^n\frac{c_i}{c}\lambda_i\right)=\frac{1}{c}\sum_{i=1}^nc_ih(T,\lambda_i).$$
In addition, $\widehat{n}\geq \frac{2h_{\lambda}(T)}{\eta}$ gives
$$c(\widehat{n}+l)(h_{\lambda}(T)-\frac{\eta}{2})\geq c(\widehat{n}+l+1)(h_{\lambda}(T)-\eta)> n(h_{\lambda}(T)-\eta).$$
As consequences, we have
$$\prod_{i=1}^n|\Gamma_i|^{c_i}\geq 2^{c(\widehat{n}+l)(h_{\lambda}(T)-\frac{\eta}{2})}\geq 2^{n(h_{\lambda}(T)-\eta)}.$$
\end{proof}

In fact, the above finite rational convex combinations of invariant measures can approximate any measure $\mu$ in $K$ while their entropy can also approximate $h_{\mu}(T)$. In other words, they
play the part of ergodic measures used in \cite{PS2007} so that we can have a local uniform separation property here for our $K$. Therefore, we naturally arrive at  the following conclusion.
\begin{Lem}\label{lem-N-K}
For any $\eta>0$, there exist $\delta^*$ and $\eps^*>0$ so that for any $\mu\in K$ and any neighborhood $F\subset M(X)\cap K$ of $\mu$, there exists $n^*_{F,\mu,\eta}$ such that
$$N(F;\delta^*,n,\eps^*)\geq 2^{n(h_{\mu}(T)-\eta)}~\textrm{if}~n\geq n^*_{F,\mu,\eta}.$$
\end{Lem}
\begin{proof}
Let $\widetilde{K}$ be the set of rational convex combination of $\{\lambda_i\}_{i=1}^n$. According to  lemma \ref{lem-rational-covex}, we see that for any $\frac{\eta}{2}>0$, there exist $\delta^*>0$ and $\eps^*>0$ so that if $\lambda\in \widetilde{K}$, then for any neighborhood $F\subset M(X)$ of $\lambda$, there exists $n^*_{F,\lambda,\frac{\eta}{2}}$ such that \eqref{N-lambda} holds with $\eta$ replaced by $\frac{\eta}{2}$. Now if $\mu\in \widetilde{K}$, we are done. Otherwise, we choose by the ergodic decomposition theorem a $\lambda\in \widetilde{K}\cap F$ such that $h_{\mu}(T)-h_{\lambda}(T)\leq\frac{\eta}{2}$. The statement is true with $n^*_{F,\mu,\eta}=n^*_{F,\nu,\frac{\eta}{2}}$.
\end{proof}

Pfister and  Sullivan used the almost specification property and the uniform separation property to obtain the following variational principle for any compact connected non-empty set $K\subset M(T,X)$:
\begin{equation}\label{variational-principle}
  \htop(T,G_K)=\inf\{h_{\mu}(T)\colon \mu\in K\}.
\end{equation}

However, if one exploits their proof carefully, it is not hard to see that, in essence, their main result can be restated as follows.
\begin{Thm}
Assume $(X,f)$ has the almost specification property. Consider any compact connected non-empty set $K\subset M(T,X)$. If $K$ satisfies that for any $\eta>0$, there exist $\delta^*$ and $\eps^*>0$ so that for any $\mu\in K$ and any neighborhood $F\subset M(X)\cap K$ of $\mu$, there exists $n^*_{F,\mu,\eta}$ such that
$$N(F;\delta^*,n,\eps^*)\geq 2^{n(h_{\mu}(T)-\eta)}~\textrm{if}~n\geq n^*_{F,\mu,\eta},$$
then the variational principle \eqref{variational-principle} holds.
\end{Thm}
In particular, our finite convex combination of invariant measures $K$ fulfills the condition. Thus we conclude the proof of theorem \ref{thm-var-prin}.

\section{Entropy for various sets}\label{proof-for-entropy}
\subsection{Proof of theorem \ref{thm-inter-reg-irreg}}
Suppose that $R_{\varphi_{1}}\cap I_{\varphi_{2}}\neq\emptyset$, then we select an $x\in R_{\varphi_{1}}\cap I_{\varphi_{2}}$. By definition, there exist two measures $\mu_{1},\mu_{2}\in V_{T}(x)$ such that $\mu_{1}\neq\mu_{2}$ and
$$\int\varphi_{1}d\mu_{1} = \int\varphi_{1}d\mu_{2},~~\int\varphi_{2}d\mu_{1} \neq \int\varphi_{2}d\mu_{2}.$$
For any $\eps>0$, according to variational principal \cite{Walters}, we select an invariant measure $\omega$ such that
$$h_{\omega}(T)>\htop(T)-\frac{\eps}{2}.$$
Then we choose $0\leq\theta\leq1$ close to $1$ such that
$$\theta h_{\omega}(T)>\htop(T)-\eps.$$
Define $\omega_{i}:=\theta\omega+(1-\theta)\mu_{i},~i=1,2.$ Then we can verify that
$$\int\varphi_{1}d\omega_{1} = \int\varphi_{1}d\omega_{2},~~\int\varphi_{2}d\omega_{1} \neq \int\varphi_{2}d\omega_{2}.$$
Define $K:=\{t\omega_{1}+(1-t)\omega_{2}|0\leq t\leq1\}$. By theorem \ref{thm-var-prin} and the affinity of entropy, we have
$$\htop(G_{K})=\min\{h_{\omega_{1}(T)},h_{\omega_{2}}(T)\}\geq \theta h_{\omega}(T)>\htop(T)-\eps.$$

On the other hand, for any $y\in G_{K}$, we have $V_{T}(y)=K$. Let $\int\varphi_{1}d\omega_{1} = \int\varphi_{1}d\omega_{2}=a$. Suppose $y\notin R_{\varphi_1}$. Then there exists a subsequence $\{u_k\}_k$ such that $\lim_k\frac{1}{u_k}\sum_{i=0}^{u_k-1}\varphi_1(T^iy)=b\neq a$. Select a subsequence $\{u_{k_n}\}$ of $u_k$ such that $\lim_n\mathcal{E}_{u_{k_n}}(y)=t\omega_{1}+(1-t)\omega_{2}$ for some $0\leq t\leq 1$. Then
$$\lim_{k}\frac{\sum_{i=0}^{u_k-1}\varphi_{1}(T^{i}(y))}{u_k}=\lim_{n}\frac{\sum_{i=0}^{u_{k_n}-1}\varphi_{1}(T^{i}(y))}{u_{k_n}}=\int\varphi_{1}d(t\omega_{1}+(1-t)\omega_{2})=a,$$
which is a contradiction.
Thus $G_{K}\subset R_{\varphi_{1}}$.

Moreover, suppose that $\lim_{k}\mathcal{E}_{l_{k}}(y)=\omega_{1}$ and $\lim_{k}\mathcal{E}_{m_{k}}(y)=\omega_{2}$. Then one can see that
$$\lim_{k}\frac{\sum_{i=0}^{l_{k}-1}\varphi_{1}(T^{i}(x))}{l_{k}}=\int\varphi_{2}d\omega_{1}\neq\int\varphi_{2}d\omega_{2}=\lim_{k}\frac{\sum_{i=0}^{m_{k}-1}\varphi_{1}(T^{i}(x))}{m_{k}}.$$
Thus $G_{K}\subset I_{\varphi_{2}}$. Hence $G_{K}\subset R_{\varphi_{1}}\cap I_{\varphi_{2}}$. By the arbitrariness of $\eps$, we see that
$$\htop(R_{\varphi_{1}}\cap I_{\varphi_{2}})=\htop(T).$$

\subsection{Proof of theorem \ref{thm-inter-level-irreg}}
Denote $h=\sup~\{h_{\mu}(T)~|~\mu\in M(T,X)~and~\int\varphi_{1} d\mu=a\}$. For any $x\in R_{\varphi_{1}}(a)\cap I_{\varphi_{2}}$, one has
$$\lim_{n}\frac{1}{n}\sum_{i=0}^{n-1}\varphi_{1}(T^{i}x)=a.$$
Then for any $\mu\in V_{T}(x)$, one has $\int\varphi_{1}d\mu=a$.
Now recall the following lemma from Bowen \cite{Bowen}.
\begin{Lem}\label{lem-Bowen}
Let $T:X\rightarrow X$ be a continuous map on a compact metric space. Set
$$QR(t)=\{x\in X~|~\exists\tau\in V_{T}(x)~\textrm{with}~h_{\tau}(T)\leq t\}.$$
Then
$$h_{top}(T,QR(t))\leq t.$$
\end{Lem}
Therefore, by lemma \ref{lem-Bowen}, we have
$$\htop(T,R_{\varphi_{1}}(a)\cap I_{\varphi_{2}})\leq h.$$
It is left to show that
\begin{equation}\label{lower-bound-of-K}
 \htop(T,R_{\varphi_{1}}(a)\cap I_{\varphi_{2}})\geq h.
\end{equation}

For any $x\in R_{\varphi_{1}}(a)\cap I_{\varphi_{2}}$, there exist $\mu_{1},\mu_{2}\in V_{T}(x)$ such that
$$\int\varphi_{1}d\mu_{1} = \int\varphi_{1}d\mu_{2}=a,~~\int\varphi_{2}d\mu_{1} \neq \int\varphi_{2}d\mu_{2}.$$
For any $\eps>0$, choose an invariant measure $\omega$ such that $\int\varphi_{1}d\omega=a$ and
$$h_{\omega}(T)>h-\frac{\eps}{2}.$$
Then we select a $0\leq\theta\leq1$ close to $1$ such that
$$\theta h_{\omega}(T)>h-\eps.$$
Let $\omega_{i}:=\theta\omega+(1-\theta)\mu_{i}~i=1,2.$ Then we can verify that
$$\int\varphi_{1}d\omega_{1} = \int\varphi_{1}d\omega_{2}=a,~~\int\varphi_{2}d\omega_{1} \neq \int\varphi_{2}d\omega_{2}.$$
Define $K:=\{t\omega_{1}+(1-t)\omega_{2}|0\leq t\leq1\}$. By theorem \ref{thm-var-prin} and the affinity of entropy, we have
$$\htop(G_{K})=\min\{h_{\omega_{1}}(T),h_{\omega_{2}}(T)\}\geq \theta h_{\omega}(T)>h-\eps.$$

On the other hand, for any $y\in G_{K}$, we have $V_{T}(y)=K$. Suppose $y\notin R_{\varphi_1}(a)$. Then there exists a subsequence $\{u_k\}_k$ such that $\lim_k\frac{1}{u_k}\sum_{i=0}^{u_k-1}\varphi_1(T^iy)=b\neq a$. Select a subsequence $\{u_{k_n}\}$ of $u_k$ such that $\lim_n\mathcal{E}_{u_{k_n}}(y)=t\omega_{1}+(1-t)\omega_{2}$ for some $0\leq t\leq 1$. Then
$$\lim_{k}\frac{\sum_{i=0}^{u_k-1}\varphi_{1}(T^{i}(y))}{u_k}=\lim_{n}\frac{\sum_{i=0}^{u_{k_n}-1}\varphi_{1}(T^{i}(y))}{u_{k_n}}=\int\varphi_{1}d(t\omega_{1}+(1-t)\omega_{2})=a,$$
which is a contradiction.
Thus $G_K\subset R_{\varphi_1}(a)$.

Moreover, suppose that $\lim_{k}\mathcal{E}_{l_{k}}(y)=\omega_{1}$ and $\lim_{k}\mathcal{E}_{m_{k}}(y)=\omega_{2}$. Then one can see that
$$\lim_{k}\frac{\sum_{i=0}^{l_{k}-1}\varphi_{2}(T^{i}(y))}{l_{k}}=\int\varphi_{2}d\omega_{1}\neq\int\varphi_{2}d\omega_{2}=\lim_{k}\frac{\sum_{i=0}^{m_{k}-1}\varphi_{2}(T^{i}(y))}{m_{k}}.$$
Thus $G_{K}\subset I_{\varphi_{2}}$. Hence
$$G_{K}\subset R_{\varphi_{1}}(a)\cap I_{\varphi_{2}}$$ and \eqref{lower-bound-of-K} is proved. By the arbitrariness of $\eps$,
$$\htop(R_{\varphi_{1}}(a)\cap I_{\varphi_{2}})=h.$$

\subsection{Proof of theorem \ref{thm-level-of-irreg}}
When $c=d$, it is solved by theorem \ref{thm-inter-level-irreg}. So we suppose $c<d$. If $X_{\varphi}(c,d)=\emptyset$, it is trivial. Thus we assume $X_{\varphi}(c,d)\neq\emptyset$. Denote
$$h=\min_{\xi=c,d}~\sup~\left\{h_{\mu}(T)~|~\int\varphi d\mu=\xi\right\}.$$
Without lose of generality, we assume that
$$\sup~\left\{h_{\mu}(T)~|~\int\varphi d\mu=c\right\}\leq \sup~\left\{h_{\mu}(T)~|~\int\varphi d\mu=d\right\}.$$
For any $x\in X_{\varphi}(c,d)$, one has $\underline{\varphi}(x)=c$. So there exist a subsequence $\{n_k\}_{k\geq1}$ of $\mathbb{N}$ such that
$$\lim_{k}\frac{1}{n_k}\sum_{i=0}^{n_k-1}\varphi(T^{i}x)=c.$$
For any subsequence $\{n_{k_{l}}\}_{l\geq1}$ of $\{n_k\}_{k\geq1}$ such that
$\lim_{l}\mathcal{E}_{n_{k_{l}}}(x)=\mu,$
one has
$$\int\varphi d\mu=c~~\textrm{and}~~\mu\in M(T,X).$$
So by lemma \ref{lem-Bowen}, we have
$$\htop(T,X_{\varphi}(c,d))\leq \sup~\left\{h_{\mu}(T)~|~\int\varphi d\mu=c\right\}= h.$$
It remains to show that
\begin{equation}\label{lower-bound-of-X-varphi}
  \htop(T,X_{\varphi}(c,d))\geq h.
\end{equation}

Since $X_{\varphi}(c,d)\neq\emptyset$, there exist $\nu_{1},\nu_{2}\in M(T,X)$ such that
$$\int\varphi d\nu_{1}=c,~\int\varphi d\nu_{2}=d,$$
and
$$h_{\nu_{1}}(T)>h-\eps,~h_{\nu_{2}}(T)>h-\eps.$$
Define $K:=\{t\nu_{1}+(1-t)\nu_{2}~|~0\leq t\leq1\}$. Then, by theorem \ref{thm-var-prin}, one sees that
$$\htop(T,G_{K})=\min\{h_{\nu_{1}}(T),h_{\nu_{2}}(T)\}>h-\eps.$$

On the other hand, for any $x\in G_{K}$, one has $V_{T}(x)=K$. Thus there exist two subsequences $\{l_{k}\}_{k}$ and $\{m_{k}\}_{k}$ such that
$$\lim_{k}\mathcal{E}_{l_{k}}(x)=\nu_{1}~\textrm{and}~\lim_{k}\mathcal{E}_{m_{k}}(x)=\nu_{2}.$$
So we have
$$\lim_{k}\frac{1}{l_{k}}\sum_{i=0}^{l_{k}-1}\varphi(T^{i}x)=c~\textrm{and}~\lim_{k}\frac{1}{m_{k}}\sum_{i=0}^{m_{k}-1}\varphi(T^{i}x)=d.$$
Moreover, for any subsequence $\{n_k\}_{k\geq}$ of $\mathbb{N}$ such that
$$\lim_{k}\mathcal{E}_{n_k}(x)=t\nu_{1}+(1-t)\nu_{2},$$
one has $$\lim_{k}\frac{1}{n_k}\sum_{i=0}^{n_k-1}\varphi(T^{i}x)=t\int\varphi d\nu_{1}+(1-t)\int\varphi d\nu_{2}\in [c,d].$$
Thus one sees that
$$\underline{\varphi}(x)=c~\textrm{and}~\overline{\varphi}(x)=d.$$
In other words, one has $G_{K}\subset X_{\varphi}(c,d)$ and \eqref{lower-bound-of-X-varphi} is proved. By the arbitrariness of $\eps$,
$$\htop(T,X_{\varphi}(c,d))=h.$$

\subsection{Proof of theorem \ref{thm-multifractal}}

Denote
$$h=\min_{\xi=c,d}~\sup~\left\{h_{\mu}(T)~|~\mu\in M(T,X)~s.t.~\int\varphi_{i}d\mu=a_i,~i=1,\ldots,n~and~\int\psi d\mu=\xi\right\}$$ and we consider two cases.

\textbf{Case 1: $c=d.$}

For any $x\in (\bigcap_{i=1}^nR_{\varphi_{i}}(a_i))\cap R_{\psi}(c)$, there exists $\mu\in V_{T}(x)$ such that
$$\int\varphi_{i}d\mu=a_i,~i=1,\ldots,n~~\textrm{and}~~\int\psi d\mu=c.$$
Thus by lemma \ref{lem-Bowen}, we see that
$$\htop\left(T,(\bigcap_{i=1}^nR_{\varphi_{i}}(a_i))\cap R_{\psi}(c)\right)\leq h.$$

On the other hand, for any $\eps>0,$ there exist $\nu\in M(T,X)$ such that
$$\int\varphi_{i}d\nu=a_i,~i=1,\ldots,n~~\int\varphi_{2}d\nu=c~~\textrm{and}~~h_{\nu}(T)>h-\eps.$$
By theorem \ref{thm-var-prin}, we have
$$\htop(T,G_{\nu})=h_{\nu}(T)>h-\eps.$$
Since $G_{\nu}\in (\bigcap_{i=1}^nR_{\varphi_{i}}(a_i))\cap R_{\psi}(c)$ and by the arbitrariness of $\eps$, we see that
$$\htop\left(T,(\bigcap_{i=1}^nR_{\varphi_{i}}(a_i))\cap R_{\psi}(c)\right)= h.$$

\textbf{Case 2: $c<d.$}

If $X_{\varphi}(c,d)=\emptyset$, there is nothing to prove. So we suppose $X_{\varphi}(c,d)\neq\emptyset$. For any $x\in (\bigcap_{i=1}^nR_{\varphi_{i}}(a_i))\cap X_{\psi}(c,d)$, there exist $\mu_{1},\mu_{2}\in V_{T}(x)$ such that
$$\int\varphi_{i}d\mu_{1} = \int\varphi_{i}d\mu_{2}=a_i,~i=1\ldots,n,~~\int\psi d\mu_{1} = c~~\textrm{and}~~\int\psi d\mu_{2}=d.$$
Then by lemma \ref{lem-Bowen}, one has
$$\htop\left(T,(\bigcap_{i=1}^nR_{\varphi_{i}}(a_i))\cap X_{\psi}(c,d)\right)\leq \min\{h_{\mu_{1}}(T),h_{\mu_{2}}(T)\}\leq h.$$

On the other hand, for any $\eps>0,$ there exist $\nu_{1},\nu_{2}\in M(T,X)$ such that
$$ \int\varphi_{i}d\nu_{1} = \int\varphi_{i}d\nu_{2}=a_i,i=1,\ldots,n,$$
$$\int\psi d\nu_{1} = c~~\textrm{and}~~\int\psi d\nu_{2}=d,$$
$$h_{\nu_{1}}(T)>h-\eps,~~\textrm{and}~~h_{\nu_{2}}(T)>h-\eps.$$
Let $K=\{t\nu_{1}+(1-t)\nu_{2}~|~0\leq t\leq1\}$. Then, by theorem \ref{thm-var-prin},
$$\htop(T,G_{K})=\min\{h_{\nu_{1}}(T),h_{\nu_{2}}(T)\}>h-\eps.$$
On the other hand, for any $y\in G_{K}$, we have $V_{T}(y)=K$. Suppose $y\notin \bigcap_{i=1}^nR_{\varphi_i}(a_i)$. Then there exist an $1\leq i\leq n$ and a subsequence $\{u_k\}_k$ such that $\lim_k\frac{1}{u_k}\sum_{i=0}^{u_k-1}\varphi_i(T^iy)=b_i\neq a_i$. Select a subsequence $\{u_{k_n}\}$ of $u_k$ such that $\lim_n\mathcal{E}_{u_{k_n}}(y)=t\omega_{1}+(1-t)\omega_{2}$ for some $0\leq t\leq 1$. Then
$$\lim_{k}\frac{\sum_{i=0}^{u_k-1}\varphi_{i}(T^{i}(y))}{u_k}=\lim_{n}\frac{\sum_{i=0}^{u_{k_n}-1}\varphi_{i}(T^{i}(y))}{u_{k_n}}=\int\varphi_{i}d(t\omega_{1}+(1-t)\omega_{2})=a_i,$$
which is a contradiction.
Thus $G_K\subset \bigcap_{i=1}^nR_{\varphi_i}(a_i)$.

Moreover, a discussion similar to the proof in theorem \ref{thm-level-of-irreg} shows that $G_K\subset X_{\psi}(c,d)$. Therefore,
$$G_{K}\subset (\bigcap_{i=1}^nR_{\varphi_{i}}(a_i))\cap X_{\psi}(c,d).$$ Thus, by the arbitrariness of $\eps$, one has
$$\htop\left(T,(\bigcap_{i=1}^nR_{\varphi_{i}}(a_i))\cap X_{\psi}(c,d)\right)=h.$$

\section{Some comments on the topological pressure}\label{about-pressure}
Let $Z\subset X$ be an arbitrary Borel set, not necessarily compact or invariant. For $s\in \R$, we define the following quantities:
\begin{eqnarray*}
  Q(Z,s,\Gamma,\psi) &=& \sum_{B_{n_i}(x_i,\eps)\in\Gamma}\exp\left(-sn_i+\sup_{x\in B_{n_i}(x_i,\eps)}\sum_{k=0}^{n_i-1}\psi(f^kx)\right), \\
  M(Z,s,\eps,N,\psi) &=& \inf_{\Gamma}Q(Z,s,\Gamma,\psi),
\end{eqnarray*}
where the infimum is taken over all finite or countable collections of the form $\Gamma=\{B_{n_i}(x_i,\eps)\}_i$, with $x_i\in X$ such that $\Gamma$ covers $Z$ and $n_i\geq N$ for all $i\geq1$. We define
$$m(Z,s,\eps,\psi)=\lim_{N\to\infty}M(Z,s,\eps,N,\psi).$$
Since $M(Z,s,\eps,N,\psi)$ is non-decreasing in $N$, the limit exists. Then we define
$$P_Z(\psi,\eps):=\inf\{s:m(Z,s,\eps,\psi)=0\}=\sup\{s:m(Z,s,\eps,\psi)=\infty\}.$$
Finally, the \emph{topological pressure} of $\psi$ on $Z$ is given by \cite{Pesin}
$$P_Z(\psi)=\lim_{\eps\to0}P_Z(\psi,\eps).$$
Our results on the topological pressure can be phrased as follows.
\begin{Thm}\label{pressure-version}
Suppose $(X,T)$ satisfies the almost specification property and $\psi\in C(X)$. Then we have the following
\begin{enumerate}
  \item For any $\varphi_{1}, \varphi_{2}\in C(X)$, if $R_{\varphi_{1}}\cap I_{\varphi_{2}}\neq\emptyset$, then $P_{R_{\varphi_{1}}\cap I_{\varphi_{2}}}(\psi)=P_X(\psi)$.
  \item For any $\varphi_{1}, \varphi_{2}\in C(X)$, if $R_{\varphi_{1}}(a)\cap I_{\varphi_{2}}\neq\emptyset$, then
$$P_{R_{\varphi_{1}}(a)\cap I_{\varphi_{2}}}(\psi)=\sup~\left\{h_{\mu}(T)+\int\psi d\mu~|~\mu\in M(T,X)~and~\int\varphi_{1} d\mu=a\right\}.$$
  \item For any real numbers $c\leq d$, either $X_{\varphi}(c,d)$ is empty or
\begin{displaymath}
P_{X_{\varphi}(c,d)}(\psi)=\min_{\xi=c,d}~\sup~\left\{h_{\mu}(T)+\int\psi d\mu~|~\mu\in M(T,X)~and~\int\varphi d\mu=\xi\right\}.
\end{displaymath}
  \item Let $\{\varphi_i\}_{i=1}^n$ and $\psi$ be continuous functions on $X$. Then for any real numbers $\{a_i\}_{i=1}^n$ and $c\leq d$, either $X_{\psi}(c,d)$ is empty or
\begin{eqnarray*}
   & & P_{(\bigcap_{i=1}^nR_{\varphi_{i}}(a_i))\cap X_{\psi}(c,d)}(\psi) \\
   &=& \min_{\xi=c,d}~\sup~\left\{h_{\mu}(T)+\int\psi d\mu~|~\mu\in M(T,X)~s.t.~\int\varphi_{i}d\mu=a_i,~i=1,\ldots,n~and~\int\psi d\mu=\xi\right\}.
\end{eqnarray*}
\end{enumerate}
\end{Thm}
\begin{proof}
We only give a sketch of the proof. Let $^KG=\{x\in X:\{\E_n(x)\}_n~\textrm{has a limit-point in}~ K \}$. Analogues to Bowen's lemma \ref{lem-Bowen}, it is shown that \cite[Theorem 4.1]{PS2007}
$$\htop(T,^KG)\leq \sup\{h(T,\rho):\rho\in K\}.$$
This inequality is generalized to the pressure by Chen-Pei \cite[Theorem 3.1]{Chen-Pei}:
\begin{equation}\label{ine-pressure}
  P_{^KG}(\psi)\leq \sup\left\{h(T,\mu)+\int \psi d\mu:\mu\in K\right\}.
\end{equation}
Moreover, using a similar discussion to the proof of \cite[Theorem 2.1]{Chen-Pei} and our theorem \ref{thm-var-prin}, we can show the following:
\begin{Thm}\label{thm-var-prin-pressure}
If $K=\{\sum_{i=1}^{n}t_{i}\lambda_{i}:t_{i}\in[0,1],\sum_{i=1}^{n}t_{i}=1,i=1,\cdots,n\}$
for some finite subset $\{\lambda_i\}_{i=1}^n\subset M(T,X)$, then
$$P_{G_K}(\psi)=\min\left\{h_{\lambda_i}(T)+\int\psi d\lambda_i,i=1,\ldots,n\right\}.$$
\end{Thm}
With the help of the inequality \eqref{ine-pressure} and theorem \ref{thm-var-prin-pressure}, we conclude the proof by a superficial modification from the languages of entropy to the ones of pressure.

\end{proof}

\section{Application to suspension flows}\label{suspension-flow}

Let $f:X\to X$ be a homeomorphism of
a compact metric space $(X,d)$. We consider a continuous roof function $\rho\colon X\mapsto(0,+\infty)$.
We define the suspension space to be
$$X^{\rho}=\{(x,s)\in X\times\R\colon 0\leq s\leq \rho(x)\},$$
where $(x,s)$ is identified with $(f(x),0)$ for all $x$. Alternatively, we can define $X^{\rho}$ to be
$X\times[0,+\infty)$, quotiented by the equivalence relation $(x,t)\sim(y,s)$ iff $(x,t)=(y,s)$ or there
exists $n\in \N$ so $(f^nx,t-\sum_{i=0}^{n-1}\rho(f^ix))=(y,s)$ or $(f^{-n}x,t-\sum_{i=0}^{n-1}\rho(f^{-i}x))=(y,s)$. Let
$\pi$
denote the quotient map from $X\times[0,+\infty)$ to $X^{\rho}$. We extend the domain of
$\pi$ to $X\times(-\inf\rho,\infty)$ by identifying points of the form $(y,-t)$ with $(f^{-1}y,\rho(y)-t)$ for
$t\in(0,\inf\rho)$. We write $(x,s)$ in place of
$\pi(x,s)$ when $-\inf\rho<s<\rho(x)$. We define the flow
$\Psi=\{g_t\}$ on $X^{\rho}$ by
$$g_t(x,s)=\pi(x,s+t).$$
To a function $\Phi:X_{\rho}\to \R$, we associate the function $\varphi:X\to\R$ by $\varphi(x)=\int_0^{\rho(x)}\Phi(x,t)dt$. Since the roof function is continuous, when $\Phi$ is continuous, so is $\varphi$. For $\mu\in M(f,X)$, we define the measure $\mu_{\rho}$ by
$$\int_{X_{\rho}}\Phi d\mu_{\rho}=\int_X\varphi d\mu/\int\rho d\mu$$
for all $\Phi\in C(X_{\rho},\R)$, where $\varphi$ is defined as above. We have $\Psi$-invariance of $\mu_{\rho}$ (i.e. $\mu(g_t^{-1}A)=\mu(A)$ for all $t\geq0$ and measurable sets $A$). The map $\mathcal{R}:M(f,X)\to M(\Psi,X_{\rho})$ given by $\mu\mapsto\mu_{\rho}$ is a bijection. Abramov's theorem \cite{Abramov,Parry-Pollicott} states that $h_{\mu_{\rho}}=h_{\mu}/\int\rho d\mu$ and hence,
$$\htop(\Psi)=\sup\{h_{\mu_{\rho}}:\mu_{\rho}\in M(\Psi,X_{\rho})\}=\sup\left\{\frac{h_{\mu}}{\int\rho d\mu}:\mu\in M(f,X)\right\},$$
where $\htop(\Psi)$ is the topological entropy of the flow.
\subsection{Definitions}
Analogous to the discrete case, for a continuous function $\Phi\colon X^{\rho}\mapsto\R$, we define
$$\underline{\Phi}(x,s)=\liminf_{T\to\infty}\frac1T\int_0^T\Phi(g_t(x,s))dt~~\textrm{and}~~\overline{\Phi}(x,s)=\limsup_{T\to\infty}\frac1T\int_0^T\Phi(g_t(x,s))dt.$$
Then we divide $X^{\rho}$ into level sets $X^{\rho}=\bigsqcup_{c\leq d}X^{\rho}_{\Phi}(\Psi,c,d)$, where
$$X^{\rho}_{\Phi}(\Psi,c,d)=\{(x,s)\in X^{\rho}\colon \underline{\Phi}(x,s)=c~\textrm{and}~\overline{\Phi}(x,s)=d\}$$
Define
$$R^{\rho}_{\Phi}(\Psi,a):=X^{\rho}_{\Phi}(\Psi,a,a),~~R^{\rho}_{\Phi}:=\bigsqcup_{a\in \mathbb{R}}R^{\rho}_{\Phi}(\Psi,a)~~\textrm{and}~~I^{\rho}_{\Phi}:=\bigsqcup_{c<d}X^{\rho}_{\Phi}(\Psi,c,d).$$
Moreover, the entropy for a flow $\Psi=\{\psi_t\}$ can also be defined analogous to the discrete case.

For any $x\in X$ and $\eps>0$, define
$$B_t(\Psi,x,\eps) = \{y\in X:d(\psi_{\tau}(x),\psi_{\tau}(y))<\eps,\tau\in[0,t)\},t\geq0.$$
Let $s\in \R$. We define the following quantities:
\begin{eqnarray*}
  Q(Z,s,\Gamma) &=& \sum_{B_{t_i}(\Psi,x_i,\eps)\in\Gamma}\exp\left(-st_i\right), \\
  M(Z,s,\eps,N) &=& \inf_{\Gamma}Q(Z,s,\Gamma),
\end{eqnarray*}
where the infimum is taken over all finite or countable collections of the form $\Gamma=\{B_{t_i}(\Psi,x_i,\eps)\}_i$, with $x_i\in X$ such that $\Gamma$ covers $Z$ and $t_i\geq N$ for all $i\geq1$.

Since $M(Z;s,\eps,N)$ is non-decreasing in $N$, we define
$$M(Z;s,\eps):=\lim_{N\rightarrow\infty}M(Z;s,\eps,N).$$
Now let
$$\htop(\Psi,Z,\eps) := \inf\{s:M(Z,s,\eps)=0\}=\sup\{s:M(Z,s,\eps)=\infty\}.$$
Then the \emph{topological entropy} of $Z$ with respect to $\Psi$ is
$$\htop(\Psi,Z):=\lim_{\eps\to0}\htop(\Psi,Z,\eps).$$

\subsection{Results}
For the suspension flows of $(X,f)$ which satisfies the specification property, Thompson proved the following variational principle \cite[Theorem 4.2]{Thompson2009}:
$$\htop(\Psi,R^{\rho}_{\Phi}(\Psi,a))=\sup\left\{h_{\mu_{\rho}}:\mu_{\rho}\in M(\Psi,X^{\rho})~\textrm{and}~\int\Phi d\mu=\alpha\right\}.$$
Thompson also proved that if $\inf_{\mu_{\rho}\in M(\Psi,X^{\rho})}\int\Phi d\mu_{\rho}<\sup_{\mu_{\rho}(\Psi,X^{\rho})}\int\Phi d\mu_{\rho}$, then
$$\htop(I^{\rho}_{\Phi},\Psi)=\htop(\Psi).$$
Here we also have various variational principles for the continuous case.
\begin{Thm}
Let $(X,d)$ be a compact metric space and $f\colon X\to X$ be a homeomorphism with the almost specification property. Let $\rho\colon X\to (0,\infty)$ be continuous. Let $(X^{\rho},\Psi)$ be the corresponding suspension flow over $X$. Then
\begin{enumerate}
  \item For any $\Phi_{1}, \Phi_{2}\in C(X^{\rho})$, $R^{\rho}_{\Phi_1}\cap I^{\rho}_{\Phi_2}$ has full entropy or is empty.
  \item For any $\Phi_{1}, \Phi_{2}\in C(X^{\rho})$, if $R^{\rho}_{\Phi_1}(\Psi,a)\cap I^{\rho}_{\Phi_2}\neq\emptyset$, then
$$\htop(\Psi,R^{\rho}_{\Phi_1}(\Psi,a)\cap I^{\rho}_{\Phi_2})=\sup~\left\{h_{\mu_{\rho}}(\Psi)~|~\mu_{\rho}\in M(\Psi,X^{\rho})~and~\int\Phi_{1} d\mu_{\rho}=a\right\}.$$
  \item \label{flow-3}For any real numbers $c\leq d$, either $X^{\rho}_{\Phi}(\Psi,c,d)$ is empty or
\begin{displaymath}
\htop(\Psi,X^{\rho}_{\Phi}(\Psi,c,d))=\min_{\xi=c,d}~\sup~\left\{h_{\mu_{\rho}}(\Psi)~|~\mu_{\rho}\in M(\Psi,X^{\rho})~and~\int\Phi d\mu_{\rho}=\xi\right\}.
\end{displaymath}
  \item Let $\{\Phi_i\}_{i=1}^n$ and $\Xi$ be continuous functions on $X^{\rho}$. Then for any real numbers $\{a_i\}_{i=1}^n$ and $c\leq d$, either $X^{\rho}_{\psi}(\Psi,c,d)$ is empty or
\begin{eqnarray*}
   & & \htop\left(\Psi,(\bigcap_{i=1}^nR^{\rho}_{\Phi_i}(\Psi,a_i)\cap X^{\rho}_{\Xi}(\Psi,c,d)\right) \\
   &=& \min_{\xi=c,d}~\sup~\left\{h_{\mu_{\rho}}(\Psi)~|~\mu_{\rho}\in M(\Psi,X^{\rho})~s.t.~\int\Phi_{i}d\mu_{\rho}=a_i,~i=1,\ldots,n~and~\int\Xi d\mu_{\rho}=\xi\right\}.
\end{eqnarray*}
\end{enumerate}
\end{Thm}
\begin{proof}
We give the proof for item \ref{flow-3}. The proof for the other three takes a similar style and needs a slight modification.

\textbf{The part for ``$\geq$''.} For an arbitrary Borel set $Z\subset X$, define $Z_{\rho}=\{(z,s):z\in Z,0\leq s<\rho((s)\}$. Thompson proved in \cite[Theorem 5.8]{Thompson2010} that, if $\beta$ is the unnique solution to the equation $P_Z(-t\rho)=0$, then $\htop(\Psi,Z_{\rho})\geq\beta$. Meanwhile, one can use the method used in the proof of theorem \ref{pressure-version} to show that
for any real numbers $c\leq d$, either $X_{\varphi,\rho}(c,d)$ is empty or
\begin{displaymath}
P_{X_{\varphi,\rho}(c,d)}(\psi)=\min_{\xi=c,d}~\sup~\left\{h_{\mu}(T)+\int\psi d\mu~|~\mu\in M(T,X)~and~\frac{\int\varphi d\mu}{\int\rho d\mu}=\xi\right\}.
\end{displaymath}
Here $X_{\varphi,\rho}(c,d)$ is defined as
$$X_{\varphi,\rho}(c,d)=\left\{x\in X:\liminf_{n\to \infty}\frac{\sum_{i=1}^{n-1}\varphi(f^ix)}{\sum_{i=1}^{n-1}\rho(f^ix)}=c~\textrm{and}~\limsup_{n\to \infty}\frac{\sum_{i=1}^{n-1}\varphi(f^ix)}{\sum_{i=1}^{n-1}\rho(f^ix)}=d\right\}.$$
Now let $A=X_{\varphi,\rho}(c,d)$ and
$$P_{A}(-h\rho)=\min_{\xi=c,d}~\sup~\left\{h_{\mu}(T)-h\int\rho d\mu~|~\mu\in M(T,X)~and~\frac{\int\varphi d\mu}{\int\rho d\mu}=\xi\right\}=0.$$
Without lose of generality, we suppose
$$\sup~\left\{h_{\mu}(T)-h\int\rho d\mu~|~\mu\in M(T,X)~and~\frac{\int\varphi d\mu}{\int\rho d\mu}=c\right\}=0.$$
Thus $h\geq\frac{h_{\mu}(T)}{\int\rho d\mu}$ for any $\mu\in M(T,X)$ such that $\frac{\int\varphi d\mu}{\int\rho d\mu}=c$. This implies that
\begin{eqnarray*}
  h &\geq& \sup\left\{\frac{h_{\mu}(T)}{\int\rho d\mu}:\mu\in M(T,X),\frac{\int\varphi d\mu}{\int\rho d\mu}=c\right\} \\
   &=& \sup\left\{h_{\mu_{\rho}}(T):\mu_{\rho}\in M(\Psi,X^{\rho}),\int\Phi d\mu_{\rho}=c\right\}.
\end{eqnarray*}
On the other hand, we have (see \cite{Thompson2010})
\begin{eqnarray*}
  \liminf_{T\to\infty}\frac1T\int_0^T\Phi(g_t(x,s))dt &=& \liminf_{n\to \infty}\frac{\sum_{i=1}^{n-1}\varphi(f^ix)}{\sum_{i=1}^{n-1}\rho(f^ix)}, \\
  \limsup_{T\to\infty}\frac1T\int_0^T\Phi(g_t(x,s))dt &=&  \limsup_{n\to \infty}\frac{\sum_{i=1}^{n-1}\varphi(f^ix)}{\sum_{i=1}^{n-1}\rho(f^ix)}.
\end{eqnarray*}
So $X^{\rho}_{\Phi}(\Psi,c,d)=A_{\rho}$, which infers that
$$\htop(\Psi,X^{\rho}_{\Phi}(\Psi,c,d))\geq h\geq\min_{\xi=c,d}\sup\left\{h_{\mu_{\rho}}(T):\mu_{\rho}\in M(\Psi,X^{\rho}),\int\Phi d\mu_{\rho}=\xi\right\}$$

\textbf{The part for ``$\leq$''.} Note that each suspension of $f$ is conjugate to the suspension of $f$ under $1$, the constant function with value $1$ \cite{Bowen-Walters}. A homeomorphism from $X^1$ to $X^{\rho}$ that
conjugates the flows is given by the map $(x,s)\mapsto(x,s\rho(x))$. So without lose of generality, we suppose that our $\rho\equiv1$.

We first observe that for any Borel set $Z\subset X^{1}$, one has
\begin{equation}\label{flow-leq-map}
  \htop(\Psi,Z)\leq \htop(\Psi^1,Z),
\end{equation}
where $\Psi^1$ is the time-$1$ map of the flow $\Psi$. Indeed, recall the definition of entropy for maps and flows. It is not hard to see that
$$C(Z;s,\eps,N)\leq M(Z,s,\eps,N)~\textrm{for any}~s\in\R,N>0,\eps>0,$$
which indicates \eqref{flow-leq-map}.

Moreover, for $(X^1,\Psi^1)$, we apply theorem \ref{thm-level-of-irreg} and get that either $X^1_{\Phi}(\Psi^1,c,d)$ is empty or
\begin{displaymath}
\htop\left(T,X^1_{\Phi}(\Psi^1,c,d)\right)=\min_{\xi=c,d}~\sup~\left\{h_{\mu}(T)~|~\mu\in M(\Psi^1,X^1)~and~\int\Phi d\mu=\xi\right\}.
\end{displaymath}
Then for any $\gamma>0$, we choose a $\nu\in M(\Psi^1,X^1)$ with $\int\Phi d\nu=c$ such that
$$\htop\left(T,X^1_{\Phi}(\Psi^1,c,d)\right)-\gamma<h_{\nu}(T).$$
For such $\nu$, we construct a measure $\mu$ by
$$\int_{X^1}\zeta(x,s)d\mu=\int_{X^1}\int_0^1\zeta(g_t(x,s))dt d\nu$$
for all $\zeta\in C(X^1,\R)$.
A straight calculation shows that
$$\int_{X^1}\zeta(x,s)d\mu=\int_{X^1}\zeta(g_t(x,s))d\mu~\textrm{for any}~t\in\R.$$
Thus $\mu\in M(\Psi,X^1)$. One also notes that \cite{Abramov,Parry-Pollicott} $h_{\mu}=h_{\nu}$. In addition,
$$\int\Phi d\mu=\int_{X^1}\int_0^1\Phi(g_t(x,s))dt d\nu=\int_0^1cdt=c.$$
Meanwhile, since our $\rho\equiv1$, a simple estimation similar to that of \cite[Lemma 5.3]{Thompson2010} shows that
$$X^1_{\Phi}(\Psi^1,c,d)=X^1_{\Phi}(\Psi,c,d).$$
So we have
$$\htop(\Psi,X^{1}_{\Phi}(\Psi,c,d))-\gamma< \sup~\left\{h_{\mu}(\Psi)~|~\mu\in M(\Psi,X^{1})~and~\int\Phi d\mu=c\right\}.$$
Similarly, we can get
$$\htop(\Psi,X^{1}_{\Phi}(\Psi,c,d))-\gamma< \sup~\left\{h_{\mu}(\Psi)~|~\mu\in M(\Psi,X^{1})~and~\int\Phi d\mu=d\right\}.$$
By the arbitrariness of $\gamma$, we obtain
$$\htop(\Psi,X^{1}_{\Phi}(\Psi,c,d))\leq \min_{\xi=c,d}~\sup~\left\{h_{\mu}(\Psi)~|~\mu\in M(\Psi,X^{1})~and~\int\Phi d\mu=\xi\right\}.$$

\end{proof}

\section*{Acknowlegements}  The research of X. Tian was  supported by National Natural Science Foundation of China (grant no. 11301088).

\end{document}